\documentclass[12pt]{amsart}
\usepackage{amsfonts, amssymb}
\usepackage{fullpage}
\usepackage{latexsym}
\usepackage{verbatim}

\DeclareMathOperator{\Aut}{Aut}

\DeclareMathOperator{\GL}{GL}

\numberwithin{equation}{section}

\newtheorem{thm}[equation]{Theorem}

\newtheorem{prop}[equation]{Proposition}
\newtheorem{conj}[equation]{Conjecture}

\theoremstyle{definition}

\newtheorem{exmp}[equation]{Example}
\theoremstyle{remark}
\newtheorem{rem}[equation]{Remark}
\theoremstyle{remark}
\newtheorem{rems}[equation]{Remarks}


\title[On Tits' Centre Conjecture]
{On Tits' Centre Conjecture \\ for Fixed Point Subcomplexes}

\author[M.\  Bate]{Michael Bate}
\address
{Department of Mathematics,
University of York,
York YO10 5DD,
United Kingdom}
\email{meb505@york.ac.uk}

\author[B.\ Martin]{Benjamin Martin}
\address
{Mathematics and Statistics Department,
University of Canterbury,
Private Bag 4800,
Christchurch 8140,
New Zealand}
\email{B.Martin@math.canterbury.ac.nz}

\author[G. R\"ohrle]{Gerhard R\"ohrle}
\address
{Fakult\"at f\"ur Mathematik,
Ruhr-Universit\"at Bochum,
D-44780 Bochum, Germany}
\email{gerhard.roehrle@rub.de}

\begin{document}

\begin{abstract}
We give a short and uniform proof of a special case of Tits' Centre Conjecture
using a theorem of J-P.\   Serre \cite{serre2} and a result from \cite{BMR}.
We consider fixed point subcomplexes $X^H$ of the building $X = X(G)$ of
a connected reductive algebraic group $G$, where
$H$ is a subgroup of $G$.
\end{abstract}

\maketitle

\section{Introduction}
\label{s:intro}
Let $G$ be a connected reductive linear algebraic group
defined over an algebraically closed field $k$.
Let $X = X(G)$ be the spherical Tits building of $G$,
cf.\ \cite{tits1}.
Recall that the simplices in $X$ correspond to
the parabolic subgroups of $G$, \cite[\S 3.1]{serre2};
for a parabolic subgroup $P$ of $G$, we let $x_P$ denote the
corresponding simplex of $X$.
The conjugation action of $G$ on itself naturally induces an action of $G$
on the building $X$, so the image of $G$ is a subgroup of the automorphism
group of $X$.
Given a subcomplex $Y$ of $X$, let $N_G(Y)$ denote
the subgroup of $G$ consisting of elements which stabilize $Y$
(in this induced action).

Recall the \emph{geometric realization} of $X$ as a bouquet of $n$-spheres.
A subcomplex $Y$ of $X$ is called \emph{convex} if whenever two points of
$Y$ (in the geometric realization) are not opposite in $X$, then  $Y$
contains the unique geodesic joining these points, \cite[\S 2.1]{serre2}.
A convex subcomplex $Y$ of $X$ is
\emph{contractible} if it has the homotopy type of a point,
\cite[\S 2.2]{serre2}.
The following is a version due to J-P.\  Serre
of the so-called ``Centre Conjecture'' by J.\ Tits,
cf.\  \cite[Lem.\ 1.2]{tits0},
\cite[\S 4]{serre1.5},  \cite[\S 2.4]{serre2}, \cite{tits2}.
This has been proved by B.\ M\"uhlherr and J.\ Tits for spherical
buildings of classical type \cite{muhlherrtits}.  The simplex referred to in the conjecture is called a {\em centre} for $Y$.

\begin{conj}
\label{conj}
Let $Y$ be a convex contractible subcomplex of $X$. Then there is a
simplex in $Y$ which is stabilized by all automorphisms of $X$ 
which stabilize $Y$.
\end{conj}

For a subgroup $H$ of $G$
let $X^H$ be the fixed point subcomplex of the action of $H$, i.e.,
$X^H$ consists of the simplices $x_P \in X$ such that $H \subseteq P$.
Thus, if $H \subseteq K \subseteq G$ are subgroups of $G$, then
we have $X^K \subseteq X^H$;
observe that $X^H$ is always
convex, cf.\ \cite[Prop.\ 3.1]{serre2}.
Our main result, Theorem \ref{mainthm}, gives a short, conceptual
proof of a special case of Conjecture \ref{conj};
namely, we consider subcomplexes of the form $Y = X^H$
for $H$ a subgroup of $G$, and we consider automorphisms only from
$N_G(Y)$.
The special case $G = \GL(V)$ in Theorem \ref{mainthm} 
generalizes the classical construction of upper and lower Loewy series,
see Remark \ref{rems2}(ii).

The initial motivation for Tits' Conjecture \ref{conj}
was a question about the existence of a certain parabolic subgroup
associated with a unipotent subgroup of a Borel subgroup of $G$
(cf.\ \cite[\S 4.1]{serre1.5}, \cite[\S 2.4]{serre2}).
This existence theorem was ultimately proved by other means,
\cite[\S 3]{boreltits2}.
In Example \ref{ex:BT} below we show that 
the existence of such a parabolic subgroup
can be viewed as a special case of Theorem \ref{mainthm}.

\section{Serre's notion of complete reducibility}
Following  Serre \cite[Def.\ 2.2.1]{serre2},
we say that a convex subcomplex $Y$ of $X$ is
\emph{$X$-completely reducible} ($X$-cr) if
for every simplex $y\in Y$ there
exists a simplex $y'\in Y$ opposite to $y$ in $X$.
The following is part of a theorem due to Serre, \cite[Thm.\ 2]{serre1.5};
see also \cite[\S 2]{serre2} and \cite{tits2}.

\begin{thm}
\label{gcr-building}
Let $Y$ be a convex subcomplex of $X$.
Then $Y$ is $X$-completely reducible if and only if
$Y$ is not contractible.
\end{thm}

The notion of convexity for subcomplexes of $X$ has the following nice
characterization in terms of parabolic subgroups due to Serre,
\cite[Prop.\ 3.1]{serre2}.

\begin{prop}
\label{convexhull}
Let $Y$ be a subcomplex of $X$.
Then $Y$ is convex if and only if whenever $P, P'$, and $Q$ are
parabolic subgroups in $G$ with $x_P, x_{P'}\in Y$ and
$Q \supseteq P \cap P'$, then $x_Q \in Y$.
\end{prop}

Note that many subcomplexes which arise naturally in the building
are fixed point subcomplexes.
For example, the apartments of $X$ are the subcomplexes $X^T$
for maximal tori $T$ of $G$ and, more generally,
the smallest convex subcomplex containing two simplices 
$x_P$ and $x_{P'}$ is $X^{P\cap P'}$.

Following  Serre \cite{serre2}, we say that a (closed) subgroup
$H$ of $G$ is \emph{$G$-completely reducible} ($G$-cr)
provided that whenever $H$ is contained in a parabolic subgroup $P$ of $G$,
it is contained in a Levi subgroup of $P$;
for an overview of this concept see for instance \cite{serre1} and
\cite{serre2}.
In the case $G = \GL(V)$ ($V$ a finite-dimensional $k$-vector space)
a subgroup $H$ is $G$-cr exactly when
$V$ is a semisimple $H$-module,
so this faithfully generalizes
the notion of complete reducibility from representation theory.
An important class of $G$-cr subgroups consists of those that are
not contained in any proper parabolic subgroup of $G$ at all
(they are trivially $G$-cr).
Following Serre, we call them
\emph{$G$-irreducible} ($G$-ir), \cite{serre2}.
As before, in the case $G = \GL(V)$, this concept
coincides with the usual notion of irreducibility.
If $H$ is a $G$-completely reducible subgroup of $G$, then $H^0$ is reductive,
\cite[Property 4]{serre1}.

Since $X^H$ is a convex subcomplex of $X = X(G)$ for
any subgroup $H$ of $G$,
Theorem \ref{gcr-building} applies in this case and we have the following result (see \cite[p19]{serre1},
\cite[\S 3]{serre2}):

\begin{thm}
\label{gcr-building2}
Let $H$ be a   subgroup of $G$. 
Then $H$ is $G$-completely reducible if and only if the subcomplex
$X^H$ is not contractible.
\end{thm}

\begin{rem}
\label{rem1}
By convention, the empty subcomplex of $X$ is not contractible.
This is consistent with Theorem \ref{gcr-building}, because
$H$ is $G$-ir if and only if $X^H = \varnothing$, and
a $G$-ir subgroup is $G$-cr.
\end{rem}

Our next result \cite[Thm.\ 3.10]{BMR}
gives an affirmative answer to a question by
Serre, \cite[p.\  24]{serre1}.
The special case when $G = \GL(V)$ is just a particular
instance of Clifford Theory.

\begin{thm}
\label{Serrequestion}
Let $N \subseteq H \subseteq G$ be   subgroups of $G$
with $N$ normal in $H$.
If $H$ is $G$-completely reducible, then so is $N$.
\end{thm}

\section{Tits' Centre Conjecture for  fixed point subcomplexes}
\label{sect:Centreconj}
Here is the main result of this note.

\begin{thm}
\label{mainthm}
Let $Y$ be a convex, contractible subcomplex of $X$.
Suppose that $Y$ is of the form $Y = X^H$ for a subgroup $H$ of $G$.
Then there is a simplex in $Y$ which is stabilized by all elements in $N_G(Y)$.
\end{thm}

\begin{proof} 
Let $M$ be the intersection of all parabolic subgroups
of $G$ corresponding to simplices in $Y$.
Since $H \subseteq M$, we have $X^M \subseteq X^H$.
But every parabolic subgroup containing $H$ contains $M$, by definition
of $M$. Hence $X^M = X^H$.  Set $K:= N_G(Y)$.
It is clear that $M$ is normal in $K$.
Since $X^K \subseteq X^M$, it suffices to show that $X^K \ne \varnothing$.
Now $Y = X^M$ is contractible, so Theorem \ref{gcr-building2} implies
that $M$ is not $G$-cr. Thus, by Theorem \ref{Serrequestion},
it follows that
$K$ is not $G$-cr and again by Theorem \ref{gcr-building2} that
$X^K$ is contractible.
In particular, $X^K$ is non-empty, by Remark \ref{rem1}.
Thus $K$ stabilizes a simplex in $X^M$, as claimed.
\end{proof}

\begin{rems}
\label{rems2}
(i).
Let $H \subseteq K \subseteq G$ be   subgroups of $G$
with $H$ normal in $K$.
Suppose that $X^H$ is contractible.
Since $H$ is normal in $K$, the latter permutes the simplices in $X^H$, and so $K \subseteq N_G(X^H)$.
It thus follows from Theorem \ref{mainthm} that
$K$ fixes a simplex in $X^H$.

(ii).
Observe that Theorem \ref{mainthm} can be viewed as
a generalization of the classical
construction of upper and lower Loewy series
in representation theory
(for definitions, see e.g., \cite{dauns}).
Let $V$ be a finite-dimensional
$k$-vector space.
Let $H \subseteq K \subseteq \GL(V)$ be   subgroups of $\GL(V)$
with $H$ normal in $K$ and suppose that $V$ is not $H$-semisimple.
Then the upper and lower Loewy series of the $H$-module $V$
are proper $K$-stable  flags in $V$, and so they
provide ``natural centres'' for the action of $K$ on the
complex $X(V)^H$, where $X(V)$ is the flag complex of $V$.

(iii).
In \cite[Prop.\ 2.11]{serre2}, J-P.\  Serre showed
that Theorem \ref{Serrequestion} is a
consequence of Tits' Centre Conjecture \ref{conj}.
So, Theorem \ref{mainthm} is just the reverse implication
of Serre's result \cite[Prop.\ 2.11]{serre2} in the special case when
Theorem \ref{Serrequestion} applies.

(iv).
Let $k_0$ be any field and let $k$ be the algebraic closure of $k_0$.
Suppose that $G$ is defined over $k_0$.  One can define what it means
for a subgroup $H$ defined over $k_0$ to be $G$-completely reducible over
$k_0$, cf. \cite[Sec.\ 5]{BMR}, \cite[Sec.~3]{serre2}.
In \cite[Thm.\ 5.8]{BMR}, it is proved that if $k_0$ is perfect,
then a subgroup $H$ is $G$-cr over $k_0$ if and only if it is $G$-cr.
Using this, one can show that the proof of Theorem \ref{mainthm} goes
through for buildings of the form $X=X(G(k_0))$.
In particular, this includes many
finite spherical buildings attached to finite groups of Lie type.

(v).
In the Centre Conjecture \ref{conj},
one considers all automorphisms of the building.
If $X=X(G)$, then in many cases, $\Aut X$ is generated
by inner and graph automorphisms of $G$,
together with field automorphisms (cf.\ \cite[Intro.]{tits1}).
We will consider graph and field automorphisms in the
setting of Theorem \ref{mainthm} in future work (see \cite[Sec.~6]{git}).
\end{rems}

Our final result gives a characterization of subcomplexes of $X$
of the form $X^H$ for a subgroup $H$ of $G$.

\begin{prop}
\label{fixedpt}
Let $Y \subseteq X$ be a subset of simplices of $X$.
Then $Y$ is a subcomplex of $X$
of the form $Y = X^H$ for some subgroup $H$ of $G$ if and only if
for every $n \in \mathbb N$, the following condition holds:
\begin{equation}
\label{fixedptcondition}
\text{if } P_1, \ldots, P_n, Q
\text{ are parabolic subgroups with }
x_{P_i} \in Y
\text{ and }
Q \supseteq \bigcap_{i=1}^n P_i, \text{ then }  x_Q \in Y.
\end{equation}
\end{prop}

\begin{proof}
First suppose that $Y = X^H$ for some subgroup $H$ of $G$.
Let $n \in \mathbb N$ and let $x_{P_1}, \ldots, x_{P_n} \in Y$.
If $Q$ is a parabolic subgroup of $G$ containing
$\cap_{i=1}^n P_i$, then $Q$ contains $H$, because each $P_i$ does,
so $x_Q \in Y$.

Conversely, suppose that condition \eqref{fixedptcondition} holds for
all $n \in \mathbb N$. Let $H$ be
the intersection of all $P$ such that $x_P \in Y$.
By the descending chain condition, we have $H = \cap_{i=1}^m P_i$
for some $m \in \mathbb N$ and some $P_i$ with $x_{P_i} \in Y$.
It follows from condition \eqref{fixedptcondition} for $n=m$ that
for any parabolic subgroup $P$ containing $H$, $x_P \in Y$,
so $X^H \subseteq Y$. It is clear from the definition of $H$ that
$Y \subseteq X^H$.
\end{proof}

\begin{rem}
Note that $Y$ is a subcomplex of $X$ precisely when
condition \eqref{fixedptcondition} holds for $n=1$. Further,
by Proposition \ref{convexhull}, $Y$ is convex if and only if
condition \eqref{fixedptcondition} holds for $n=2$.
\end{rem}

As indicated in the Introduction, a fundamental theorem of Borel and Tits
on unipotent subgroups of Borel subgroups of $G$ \cite[\S 3]{boreltits2}
yields a key example for Theorem \ref{mainthm}.

\begin{exmp}
\label{ex:BT}
Let  $U$ be a non-trivial 
unipotent subgroup of $G$ contained in a Borel subgroup $B$ of $G$.
Let $Y = X^U$. 
Note that $U$ is not $G$-cr; for if $U$ is contained in a
Borel subgroup $B^-$ opposite to $B$, then $U$ is contained in the
maximal torus $B^- \cap B$ of $G$, which is absurd.
So $Y$ is contractible, by Theorem \ref{gcr-building2}.
Thus, by Theorem \ref{mainthm}, $N_G(U)$ stabilizes a simplex in $Y$,
i.e., there is a parabolic subgroup $P$ of $G$ containing  $N_G(U)$.
Now, the construction of Borel and Tits in \cite{boreltits2}
yields such a parabolic subgroup $P$ which enjoys additional properties; 
for example, it is stabilized by automorphisms of $G$ which stabilize $U$.
The framework for $G$-complete reducibility developed
in \cite{BMR} and subsequent papers allows one to associate
such \emph{canonical} parabolic subgroups to all non-$G$-cr subgroups
of $G$, see \cite[Sec.\ 5]{git}.
\end{exmp}


\bigskip
{\bf Acknowledgements}:
The authors acknowledge the financial support of EPSRC Grant EP/C542150/1
and Marsden Grant UOC0501.
Part of the paper was
written during a visit of the third author to the Max-Planck-Institute
for Mathematics in Bonn.


\end{document}